\documentclass{article}
\usepackage[top=30truemm,bottom=30truemm,left=25truemm,right=25truemm]{geometry}
\usepackage[dvipdfmx]{graphicx}
\usepackage{amsmath, amssymb}
\usepackage{amsthm}
\usepackage{autobreak}
\usepackage{amsfonts}
\usepackage{mathdots}
\usepackage{fancybox}
\usepackage{esint}
\usepackage{mathrsfs}
\usepackage{mathtools}
\usepackage{ascmac}
\usepackage{comment}
\usepackage{proof}
\usepackage{setspace}
\theoremstyle{definition}
\newtheorem{theorem}{Theorem}
\newtheorem{definition}[theorem]{Definition}
\newtheorem{lemma}[theorem]{Lemma}

\newtheorem{proposition}[theorem]{Proposition}
\newtheorem{remark}{Remark}
\usepackage{amscd}
\usepackage[all]{xy}
\usepackage{color}
\usepackage{cases}
\usepackage{here}
\usepackage{bm}
\usepackage{authblk}
\title{Infinitely many solutions for doubly critical variable-order $p(x)$-Choquard-Kirchhoff type equations via the concentration compactness method
}
\author{Masaki Sakuma\thanks{Email: masakisakuma0110@gmail.com}}
\affil{Graduate School of Mathematical Sciences,\\ The University of Tokyo, Meguro-ku, Tokyo, Japan}

\begin{document}
\maketitle
\begin{abstract}
In this paper, we prove the existence of infinitely many solutions of a doubly critical Choquard-Kirchhoff type equation
\begin{equation*}
\begin{split}
&M(\mathcal{E}[u])(-\Delta)_{p(\cdot,\cdot)}^{s(\cdot,\cdot)}u+V(x)|u|^{p(x)-2}u-\varepsilon_W W(x)|u|^{q_W(x)-2}u \\
&\quad\quad=\left(\int_{\mathbb{R}^N} \frac{|u(y)|^{r(y)}}{|x-y|^{\alpha(x,y)}}dy\right) |u|^{r(x)-2}u + |u|^{q_c(x)-2}u\quad\text{in $\mathbb{R}^N$}
\end{split}
\end{equation*}
with variable smoothness. The exponents $r(\cdot)$ and $q_c(\cdot)$ are allowed to have critical growths at the same time in the sense of Hardy-Littlewood-Sobolev inequality and Sobolev inequality, respectively. In the course of consideration, we formulate a new Sobolev type embedding theorem for the Slobodeckij spaces with variable smoothness and extend the concentration compactness lemma to the variable exponent case in the form involving some nonlocal quantities. To obtain a sequence of solutions, we use a critical point theorem based on Krasnoselskii's genus. 

\vspace{1ex}\par
\end{abstract}
\section{Introduction}
In this paper, we consider a type of critical variable-order $p(x)$-Choquard-Kirchhoff equations of the form
\begin{equation}\label{cho}
\begin{split}
&M(\mathcal{E}[u])(-\Delta)_{p(\cdot,\cdot)}^{s(\cdot,\cdot)}u+V(x)|u|^{p(x)-2}u-\varepsilon_W W(x)|u|^{q_W(x)-2}u \\
&\quad\quad=\int_{\mathbb{R}^N} \frac{|u(x)|^{r(x)-2} |u(y)|^{r(y)}u(x)}{|x-y|^{\alpha(x,y)}}dy+ |u|^{q_c(x)-2}u\quad\text{in $\mathbb{R}^N$},
\end{split}
\end{equation}
where $N\geq 1$ is a natural number, $(-\Delta)_{p(\cdot,\cdot)}^{s(\cdot,\cdot)}$ is the variable-order $p(x)$-Laplace operator given by the Cauchy principal value
\[
(-\Delta)_{p(\cdot,\cdot)}^{s(\cdot,\cdot)}u(x)\coloneqq \lim_{\varepsilon\to +0}\int_{\mathbb{R}^N\setminus B_\varepsilon (x)}\frac{|u(x)-u(y)|^{p(x,y)-2}(u(x)-u(y))}{|x-y|^{N+s(x,y)p(x,y)}}dy
\]
and $\mathcal{E}$ is the corresponding energy functional
\[
\mathcal{E}[u]\coloneqq \int_{\mathbb{R}^N} \int_{\mathbb{R}^N}\frac{|u(x)-u(y)|^{p(x,y)}}{p(x,y) |x-y|^{N+s(x,y)p(x,y)}}dxdy
\]
defined on the Sobolev space $W^{s(x,y),p(x,y)}(\mathbb{R}^N)$ with variable order and variable exponent. We will deal with various variable-exponent function spaces, the specific definitions of which are given in the next section. Note that in the case of variable-order Sobolev spaces, the smoothness exponent and the integrability exponent have two argument. Some literature adopts the notation $s(\cdot,\cdot)$ and $p(\cdot,\cdot)$ to clarify this. In this paper, we write $s(x,y)$ and $p(x,y)$ using variables $x,y\in\mathbb{R}^N$ explicitly. However, we also consider exponents with one argument $x\in\mathbb{R}^N$ corresponding to the diagonals and write $p(x)=p(x,x)$, $s(x)=s(x,x)$ in an abuse of notation unless there is confusion. 
In addition, throughout this paper, $B_r(x)$ denotes the open ball with radius $r$ centered at $x$ in $\mathbb{R}^N$; $C$, $C'$, $C_i$ and $C_i'$ represent various positive constants; $\|\cdot\|_q$ denotes the $L^q$ norm; $f_{+}$ [resp. $f_{-}$] denotes the positive part [resp. negative part] of a function $f$; $\chi_A$ denotes the indicator function of $A$. \par

We denote $\Phi^{+}=\sup\Phi$ and $\Phi^{-}=\inf\Phi$ for various functions on various domains. We say a function $\Psi:\mathbb{R}^N\times \mathbb{R}^N\to\mathbb{R}$ is \textit{symmetric} if and only if $\Psi(x,y)=\Psi(y,x)$ for any $x,y\in \mathbb{R}^N$. The exponents that appear in the equation \eqref{cho} satisfy the following conditions:
\begin{enumerate}
\item[(S)] $s:\mathbb{R}^N\times \mathbb{R}^N\to\mathbb{R}$ is a symmetric and uniformly continuous function with $0<s^{-}\leq s^{+}<1$.
\item[(P)] $p:\mathbb{R}^N\times \mathbb{R}^N\to\mathbb{R}$ is a symmetric and uniformly continuous function with $1<p^{-}\leq p^{+}<N/s^{+}$. 
\item[(QC)] $q_c:\mathbb{R}^N\to\mathbb{R}$ is a uniformly continuous function with $p^{+}<q_c^{-}\leq q_c(x)\leq p_s^*(x)$ ($x\in\mathbb{R}^N$) and has critical growth, i.e., $\{x \in\mathbb{R}^N\mid q_c(x)=p_s^*(x)\}\neq\emptyset$, and the continuous embedding $W^{s(x,y),p(x,y)}(\mathbb{R}^N)\hookrightarrow L^{q_c(x)}(\mathbb{R}^N)$ holds, where $p_s^*(x)$ denotes the critical Sobolev exponent $\displaystyle \frac{p(x)N}{N-s(x)p(x)}$. 
\end{enumerate}
$q_c$ essentially plays the role of a substitute for the critical Sobolev exponent $p_s^*$. In the case of constant exponents, there is no problem in letting $q_c\equiv p_s^*$. On the other hand, in the case of variable exponents, the situation is complicated because the critical Sobolev embedding $W^{s(x,y),p(x,y)}(\mathbb{R}^N)\hookrightarrow L^{p_s^*(x)}(\mathbb{R}^N)$ does not always hold. Some sufficient conditions for $W^{s(x,y),p(x,y)}(\mathbb{R}^N)\hookrightarrow L^{q_c(x)}(\mathbb{R}^N)$ to hold will be considered later.

\begin{enumerate}
\item[(QW)] $q_W:\mathbb{R}^N\to\mathbb{R}$ is a continuous function with $1<q_W^{-}\leq q_W^{+}<p^{-}$. 
\item[(A)] $\alpha:\mathbb{R}^N\to\mathbb{R}$ is a continuous function with $0<\alpha^{-}\leq \alpha^{+}<N$.
\end{enumerate}
In the equation \eqref{cho}, two-argument function $\alpha(\cdot,\cdot)$ is defined as $\displaystyle \alpha(x,y)=\frac{\alpha(x)+\alpha(y)}{2}$. In addition, we define $\displaystyle\sigma_\alpha(x)\coloneqq \frac{2N}{2N-\alpha(x)}$.

\begin{enumerate}
\item[(R)] $r\in C^1(\mathbb{R}^N;\mathbb{R})$ satisfies $\displaystyle\frac{p^{+}}{\sigma_\alpha^{-}}<r(x)\leq \frac{q_c(x)}{\sigma_\alpha^{+}}$, $r^{-}>1$ and $\displaystyle r(x)\to r_\infty$ as $|x|\to\infty$
\item[(D)] $p(x,y)\to p_\infty <r_\infty$, $s(x,y)\to s_\infty$, $q_c(x)\to (q_c)_\infty$ as $|x|,|y|\to\infty$. 
\end{enumerate}
The Kirchhoff function $M\in C([0,\infty);\mathbb{R})$, the potential function $V\in L^{\frac{q_c(x)}{q_c(x)-p(x)}}_\mathrm{uloc}(\mathbb{R}^N)$ and the weight function $W\in L^{\frac{q_c(x)}{q_c(x)-q_W(x)}}(\mathbb{R}^N)$ satisfies 
\begin{enumerate}
\item[(M)] $\inf M>0$ and there exists $\theta_M\in [1,2r^{-})$ such that $\theta_M\mathcal{M}(t)\geq M(t)t$ ($t\geq 0$), where $\mathcal{M}(t)\coloneqq\int_0^t M(\tau)d\tau$. 
\item[(V)] $\tau_0\coloneqq \inf V>0$. 
\item[(W)] $W(x)>0$ for any $x\in\mathbb{R}^N$. 
\end{enumerate}
The Kirchhoff function $M$ which satisfies the condition (M) is called non-degenerate. A typical example of such functions is $M(t)=c_1 t^{\theta_M-1}+c_0$ for some constants $c_0,c_1>0$. \par

Our main theorem is stated as follows.
\begin{theorem}
Assume (S)-(W). There exists $\varepsilon>0$ such that if $0<\varepsilon_W<\varepsilon$, then the equation \eqref{cho} has infinitely many weak solutions.
\end{theorem}
The Choquard-Pekar equation
\begin{equation}\label{Pekar}
-\Delta u+u=\left(\frac{1}{|x|}\ast |u|^2\right)u\quad\text{in $\mathbb{R}^3$}
\end{equation}
first appeared in the description of the quantum theory about polaron by Pekar \cite{Pekar} in 1954 and is named after Choquard, who suggested this type of equation to describe one-component plasmas in 1976. Penrose \cite{Hartree-Fock} also dealt with this modeling related to a Hartree approximation. Mathematically, a natural general form of \eqref{Pekar} is regarded as
\begin{equation}\label{mathchoquard}
-\Delta u+Vu=\left(|x|^{-(N-\alpha)}\ast |u|^r\right)|u|^{r-2}u\quad\text{in $\mathbb{R}^N$}.
\end{equation}
When $V\equiv 1$ and $r$ is between upper and lower critical exponents in the sense of the Hardy-Littlewood-Sobolev inequality, that is, $\frac{N+\alpha}{N}<r<\frac{N+\alpha}{N-2}$, the ground state solutions exist as a consequence of the mountain pass lemma or the method of the Nehari manifold. Their qualitative properties and decay asymptotics are studied in \cite{Moroz_2}. On the other hand, in the critical case, if $V\equiv 1$, \eqref{mathchoquard} turns out to have no nontrivial solution by the Pohozaev identity. Moroz and Schaftingen \cite{Moroz_lower_critical} gave a sufficient condition concerning non-constant $V$ for the existence of ground states in the lower critical case. Several works, including \cite{doubly_critical_Seok,doubly_critical_new,doubly_critical_fractional,doubly_critical_infinitely_many}, showed the existence of nontrivial solutions of doubly critical equations such as
\[
-\Delta u+u=\left(|x|^{-(N-\alpha)}\ast F(u)\right)F'(u)\quad\text{in $\mathbb{R}^N$},
\]
where $F(u)=|u|^{\frac{N+\alpha}{N}}+|u|^{\frac{N+\alpha}{N-2}}$ up to constant factors, and its extension to a fractional version. There are also several studies on the doubly critical cases in the sense of including a local critical term and a nonlocal critical term. For example, Cai and Zhang \cite{doubly_critical_Brezis_Nirenberg} dealt with the Brezis-Nirenberg type problem
\[
-\Delta u-\lambda u=\alpha |u|^{2^*-2}u+\beta (|x|^{-(N-\alpha)}\ast |u|^{\frac{N+\alpha}{N-2}}) |u|^{\frac{N+\alpha}{N-2}-2}u \quad \text{in }\Omega,
\]
where $\Omega$ is a bounded domain and $\frac{N+\alpha}{N}$ means the upper critical exponent. Another example is \cite{Li_Li}, which obtained a ground state solution to the Choquard type equation
\[
-\Delta u+ u=(|x|^{-(N-\alpha)}\ast |u|^{\frac{N+\alpha}{N}}) |u|^{\frac{N+\alpha}{N}-2}u+|u|^{2^*-2}u+g(u) \quad \text{in }\mathbb{R}^N
\]
with a local critical term and a lower critical nonlocal term by using the Pohozaev manifold method. Concerning other recent results on classical Choquard equations, see \cite{Moroz} and references therein. $p$-fractional versions and Kirchhoff type problems have also increased over the past few years. 
\cite{Kirchhoff} obtained infinitely many solutions of the Schr\"{o}dinger-Choquard-Kirchhoff equation
\begin{equation*}
M(\|u\|_{D^{s,p}}^{p})(-\Delta )_{p}^{s} u+V(x)|u|^{p-2}u=\lambda (|x|^{-(N-\alpha)}\ast |u|^{\frac{p(N+\alpha)}{2(N-ps)}}) |u|^{\frac{p(N+\alpha)}{2(N-ps)}-2}
u+\beta k(x)|u|^{q-2}u \quad \text{in }\mathbb{R}^{N},
\end{equation*}
where $M$ is a non-degenerate Kirchhoff
function and $V\in C(\mathbb{R}^{N})$ satisfies $\inf V>0$ by the symmetric mountain pass lemma. In addition, \cite{Kirchhoff_2} studied the multiplicity of the solutions of the Schr\"{o}dinger-Choquard-Kirchhoff equation under some conditions for the growth of the nonlinearity.\par

Moreover, elliptic equations with variable exponents have also gradually attracted attention in recent years, with applications appearing in models such as electrotheological fluids \cite{fluid} and image restoration \cite{image}. As for the constant smoothness case, especially $s=1$, there is a relatively large number of prior studies on critical Choquard type equations with variable integrability exponents, such as \cite{Zhang,Fu_Zhang,critical_fractional_p(x),critical_p(x)_Kirchhoff}. Additionally, \cite{Maia} obtained some existence and multiplicity results for weak solutions of $p(x)$-Choquard-Kirchhoff equation 
\[
M(\mathcal{E}[u])((-\Delta)_{p(x)}u+V(x)|u|^{p(x)-2}u)=\lambda |u|^{q(x)-2}u+\int_{\mathbb{R}^N}\frac{|u(y)|^{r(y)}}{|x-y|^{\alpha(x,y)}} |u(x)|^{r(x)-2}u(x)dy \quad \text{in }\mathbb{R}^{N},
\]
where $|u(x)|^{r(x)}$ is allowed to have the upper critical growth. However, there is still little literature written on variable order (variable smoothness) critical Choquard type equations. 

\section{Preliminaries}
First, we briefly introduce the general theory of the Sobolev spaces with variable exponents. Let $\Omega\subset\mathbb{R}^N$ be a domain with smooth boundary. We define
\[
C_{+}(\bar{\Omega})\coloneqq \{\Phi\in C(\bar{\Omega};\mathbb{R})\mid 1<\Phi^{-}\leq \Phi^{+}<\infty\},
\]
where $\Phi^{-}\coloneqq\inf\Phi$ and $\Phi^{+}\coloneqq\sup\Phi$. \par
For $p\in C_{+}(\bar{\Omega})$, we define the Lebesgue space with variable exponent $p(\cdot)$ (the Nakano space) by
\[
L^{p(\cdot )}(\Omega)=L^{p(x)}(\Omega)\coloneqq\{u:\Omega\to\mathbb{R}\mid \text{$u$ is measurable and }\int_\Omega |u(x)|^{p(x)}dx<\infty\}
\]
endowed with the Luxemburg norm
\[
\|u\|_{p(x)}\coloneqq\inf\left\{\lambda>0\;\middle|\; \int_\Omega \left|\frac{u(x)}{\lambda}\right|^{p(x)}dx\leq 1\right\}.
\]
This space is a separable and uniformly convex Banach space. We also consider the modular map $\varrho_p: L^{p(x)}(\Omega)\to\mathbb{R}$ defined by
\[
\varrho_p(u)\coloneqq \int_{\Omega}|u(x)|^{p(x)}dx.
\]
There is an important relation between the modular and the corresponding norm. 
\begin{proposition}
For $u\in L^{p(x)}(\Omega)$, we have
\begin{enumerate}
\item[(i)] $\|u\|_{p(x)}>1\;\text{[resp. $=1$, $<1$]}\Leftrightarrow \varrho_{p}>1\; \text{[resp. $=1$, $<1$]}$;
\item[(ii)] $\|u\|_{p(x)}\geq 1\Rightarrow \|u\|_{p(x)}^{p^{-}}\leq \varrho_{p}(u)\leq \|u\|_{p(x)}^{p^{+}}$;
\item[(iii)] $\|u\|_{p(x)}\leq 1\Rightarrow \|u\|_{p(x)}^{p^{+}}\leq \varrho_{p}(u)\leq \|u\|_{p(x)}^{p^{-}}$. 
\end{enumerate}
In particular, the convergence with respect to the modular is equivalent to the norm convergence in $L^{p(x)}(\Omega)$, i.e., for any $\{u_n\}\subset L^{p(x)}(\Omega)$, we have
\[
\|u_n\|_{p(x)}\to 0\Longleftrightarrow \varrho_{p}(u_n)\to 0\quad\text{and}\quad \|u_n\|_{p(x)}\to \infty\Longleftrightarrow \varrho_{p}(u_n)\to \infty.
\]
\end{proposition}
The condition (i) is called the norm-modular unit ball property. The conditions (ii) and (iii) implies a useful estimate
\[
\|u\|_{p(x)}^{p^{-}}-1\leq \varrho_{p}(u)\leq \|u\|_{p(x)}^{p^{+}}+1.
\]
These properties also hold with the norm and the modular of $W^{s(x,y),p(x,y)}$ and $W^{s(x,y),p(x,y)}_V$ defined later instead of those of $L^{p(x)}$. \par
To handle several different variable exponents, we need the following inequalities.
\begin{proposition}
Let $p\in C_{+}(\bar{\Omega})$ and $q\in C(\bar{\Omega})$. Assume $pq\in C_{+} (\bar{\Omega})$ and $u\in L^{p(x)q(x)}(\Omega)$. Then, 
\begin{enumerate}
\item $\|u\|_{p(x)q(x)}\geq 1\Rightarrow \|u\|_{p(x)q(x)}^{q^{-}}\leq \| |u|^q\|_{p(x)}\leq \|u\|_{p(x)q(x)}^{q^{+}}$;
\item $\|u\|_{p(x)q(x)}\leq 1\Rightarrow \|u\|_{p(x)q(x)}^{q^{+}}\leq \| |u|^q\|_{p(x)}\leq \|u\|_{p(x)q(x)}^{q^{-}}$.
\end{enumerate}
\end{proposition}
The H\"{o}lder type inequality also holds for variable-exponent spaces. Let $p'(x)$ denote the H\"{o}lder conjugate exponent of $p(x)$, i.e., $1/p(x)+1/p'(x)=1$. 
\begin{proposition}
Let $p\in C_{+}(\bar{\Omega})$. For any $u\in L^{p(x)}(\Omega)$ and $v\in L^{p'(x)}(\Omega)$, we have
\[
\left|\int_\Omega uv dx\right|\leq \left(\frac{1}{p^{-}}+ \frac{1}{(p')^{-}}\right)\|u\|_{p(x)}\|v\|_{p'(x)}\leq 2 \|u\|_{p(x)}\|v\|_{p'(x)}.
\]
\end{proposition}
\begin{remark}
From this, we can deduce that if $\displaystyle \frac{1}{q(x)}=\frac{1}{p_1(x)}+ \frac{1}{p_2(x)}$, then for any $u_1\in L^{p_1(x)}$ and $u_2\in L^{p_2(x)}$, we have
\[
\|u_1u_2\|_{q(x)}\leq\left(\left(\frac{q}{p_1}\right)^{+}+ \left(\frac{q}{p_2}\right)^{+}\right)\|u_1\|_{p_1(x)}\|u_2\|_{p_2(x)}.
\]
\end{remark}

For details and proof, see, e.g., \cite{Diening,Fan_Lp(x)}. \par
The uniformly local Lebesgue space $L^{p(x)}_{\mathrm{uloc}}(\mathbb{R}^N)$ with variable exponent $p\in C_{+}(\mathbb{R}^N)$ is defined as follows.
\[
L^{p(x)}_{\mathrm{uloc}}(\mathbb{R}^N)\coloneqq\{u:\mathbb{R}^N\to\mathbb{R}:\text{measurable}\mid \forall r>0;\sup_{x\in\mathbb{R}^N}\| u|_{B_r(x)}\|_{L^{p(x)}(B_r(x))}<\infty\}.
\]

As for variable-exponent Sobolev spaces, we also treat variable exponents with two arguments in a similar way. For symmetric continuous functions $p:\bar{\Omega}\times \bar{\Omega} \to (1,\infty)$ and $s: \bar{\Omega}\times \bar{\Omega} \to (0,1)$ such that $0<s^{-}\leq s^{+}<1$ and $s^{+}p^{+}<N$ and a continuous function $\beta\in C_{+}(\bar{\Omega})$, the fractional Sobolev (Sobolev-Slobodeckij) space $W^{s(x,y),\beta(x),p(x,y)}(\Omega)$ with variable order and variable exponents is defined to be
\[
\left\{u\in L^{\beta(x)}(\Omega)\mid \exists\lambda>0\text{ s.t. }\int_\Omega\int_\Omega\frac{|u(x)-u(y)|^{p(x,y)}}{\lambda^{p(x,y)} |x-y|^{N+s(x,y)p(x,y)}}dxdy<\infty\right\}.
\]
We set
\[
[u]_{s(x,y),p(x,y);\Omega}\coloneqq\inf\left\{\lambda>0\mid \int_\Omega\int_\Omega\frac{|u(x)-u(y)|^{p(x,y)}}{\lambda^{p(x,y)} |x-y|^{N+s(x,y)p(x,y)}}dxdy<1\right\}
\]
and equip $W^{s(x,y),\beta(x),p(x,y)}(\Omega)$ with the norm
\[
\|u\|_{W^{s(x,y),\beta(x),p(x,y)}(\Omega)}\coloneqq \|u\|_{L^{\beta (x)}(\Omega)}+ [u]_{s(x,y),p(x,y);\Omega}.
\]
Let $p(x)=p(x,x)$. $\|\cdot\|_{W^{s(x,y),p(x),p(x,y)}(\Omega)}$ is equivalent to the Luxemburg norm corresponding to the modular
\[
\varrho_{W^{s,p}}(u)= \int_\Omega\int_\Omega\frac{|u(x)-u(y)|^{p(x,y)}}{|x-y|^{N+s(x,y)p(x,y)}}dxdy+ \int_\Omega |u|^{p(x)}dx.
\]
In the present paper, we deal with the case $\Omega=\mathbb{R}^N$. For $\Omega=\mathbb{R}^N$, we omit the subscript ``$\Omega$'' for the sake of notational simplicity. When $\displaystyle s(x,y)=\frac{s(x)+s(y)}{2}$ and $\displaystyle p(x,y)=\frac{p(x)+p(y)}{2}$ holds for some one-argument functions $s(\cdot),p(\cdot)\in C^{+}(\mathbb{R}^N)$, we write
\[
W^{s(x,y),\beta(x),p(x,y)}(\Omega)= W^{s(x),\beta(x),p(x)}(\Omega),
\]
\[
[u]_{s(x,y),p(x,y);\Omega} = [u]_{s(x),p(x);\Omega}.
\]
When $\beta(x)=p(x)\, (=p(x,x))$, we write $W^{s(x,y),\beta(x),p(x,y)}(\Omega)= W^{s(x,y),p(x,y)}(\Omega)$. \par

Next, we introduce the $p(x)$-generalization of the Hardy-Littlewood-Sobolev inequality (see \cite{variable_exponent_HLS}), which plays an important role in dealing with Choquard type equations with variable exponents. 
\begin{proposition}\label{HLS}
Let $p_1,p_2\in C^{+}(\mathbb{R}^N)$ and $\alpha:\mathbb{R}^N\times \mathbb{R}^N\to\mathbb{R}$ be a continuous function such that $0<\alpha^{-}\leq\alpha^{+}<N$. If 
\[
\frac{1}{p_1(x)}+ \frac{\alpha(x,y)}{N}+\frac{1}{p_2(y)}=2\quad (\forall x,y\in\mathbb{R}^N),
\]
then there exists a constant $C>0$ such that for any $f_1\in L^{p_1^{+}} (\mathbb{R}^N)\cap L^{p_1^{-}} (\mathbb{R}^N)$ and for any $f_2\in L^{p_2^{+}} (\mathbb{R}^N)\cap L^{p_2^{-}} (\mathbb{R}^N)$, we have
\[
\left|\int_{\mathbb{R}^N} \int_{\mathbb{R}^N}\frac{f_1(x)f_2(y)}{|x-y|^{\alpha(x,y)}}dxdy\right|\leq C(\|f_1\|_{L^{p_1^{+}}} \|f_2\|_{L^{p_2^{+}}}+ \|f_1\|_{L^{p_1^{-}}} \|f_2\|_{L^{p_2^{-}}}).
\]
\end{proposition}
\begin{remark}
Let us fix $\alpha:\mathbb{R}^N\times \mathbb{R}^N \to \mathbb{R}$ such that $0<\alpha^{-}\leq \alpha^{+}<N$, and $r\in C^{+}(\mathbb{R}^N)$. If there exists $\sigma\in C^{+}(\mathbb{R}^N)$ such that for any $x,y\in\mathbb{R}^N$, we have
\[
\frac{1}{\sigma(x)}+\frac{\alpha(x,y)}{N}+\frac{1}{\sigma(y)}=2
\]
and that for any $x\in\mathbb{R}^N$, we have
\[
p(x)\leq r(x)\sigma^{-}\leq r(x)\sigma^{+}\leq p_s^*(x),
\]
then by applying Proposition \ref{HLS} to the case where $p_1=p_2=\sigma$ and $f_1=f_2=|u|^{r(x)}$ for any $u\in W^{s(x,y),p(x,y)}(\mathbb{R}^N)\hookrightarrow L^{r(x)\sigma^{-}} (\mathbb{R}^N)\cap L^{r(x)\sigma^{+}} (\mathbb{R}^N)$, we obtain
\begin{align*}
&\phantom{=}\int_{\mathbb{R}^N} \int_{\mathbb{R}^N}\frac{|u(x)|^{r(x)}|u(y)|^{r(y)}}{|x-y|^{\alpha(x,y)}}dxdy\\ 
&\leq C\max\{\|u\|_{r(x)\sigma^{-}}^{2 r^{+}}, \|u\|_{r(x)\sigma^{-}}^{2 r^{-}}, \|u\|_{r(x)\sigma^{+}}^{2 r^{+}}, \|u\|_{r(x)\sigma^{+}}^{2 r^{-}}\}.
\end{align*}
Furthermore, inferring as in \cite{variable_exponent_HLS}, we can see the functional
\[
K: W^{s(x,y),p(x,y)}(\mathbb{R}^N)\to\mathbb{R}; u\mapsto \int_{\mathbb{R}^N} \int_{\mathbb{R}^N}\frac{|u(x)|^{r(x)}|u(y)|^{r(y)}}{2r(x)|x-y|^{\alpha(x,y)}}dxdy
\]
is Fr\'{e}chet differentiable with derivative
\[
K'[u]v=\int_{\mathbb{R}^N} \int_{\mathbb{R}^N}\frac{|u(x)|^{r(x)}|u(y)|^{r(y)-2}v(y)}{|x-y|^{\alpha(x,y)}}dxdy.
\]
In addition, the same formula holds with $W^{s(x,y),p(x,y)}_V(\mathbb{R}^N)$ defined later instead of $W^{s(x,y),p(x,y)}(\mathbb{R}^N)$. 
\end{remark}
\begin{definition}
If we have $\{x\in\mathbb{R}^N\mid r(x)\sigma_\alpha^{-}=p(x)\}\neq\emptyset$, then we say that $K$ described above has the lower critical growth in the sense of the Hardy-Littlewood-Sobolev inequality. Similarly, if $\{x\in\mathbb{R}^N\mid r(x)\sigma_\alpha^{+}=p_s^*(x)\}\neq\emptyset$, then we say that $K$ has the upper critical growth in the sense of the Hardy-Littlewood-Sobolev inequality. 
\end{definition}
Let us notice that the condition (R) and (QC) allows $K$ to have the upper critical growth.\par
We now define the function space we work in. Define a modular
\[
\varrho_E(u)\coloneqq \int_{\mathbb{R}^N} \int_{\mathbb{R}^N}\frac{|u(x)-u(y)|^{p(x,y)}}{|x-y|^{N+s(x,y)p(x,y)}}dxdy+ \int_{\mathbb{R}^N} V_{+}|u|^{p(x)}dx
\]
for any measurable function $u$ and the corresponding Luxemburg norm
\[
\|u\|=\|u\|_E\coloneqq\inf\left\{\lambda>0: \varrho_{E}(u/\lambda)\leq 1 \right\}.
\]
The modular space
\begin{align*}
E&=W^{s(x,y),p(x,y)}_V(\mathbb{R}^N)\\
&\coloneqq \{u:\text{measurable}\mid \lim_{\lambda\to\infty}\varrho_{E}(u/\lambda)=0\} \\
&=\{u:\text{measurable}\mid \exists\lambda>0\text{ s.t. }\varrho_{E}(u/\lambda)<\infty\}
\end{align*}
equipped with the norm $\|\cdot\|_E$ is uniformly convex Banach space (due to Lemma 2.4.16 and Theorem 2.4.14 in \cite{Diening}). \par
The energy functional $I:E\to\mathbb{R}$ associated with the equation \eqref{cho} is
\begin{equation}\label{functional}
\begin{split}
I[u]&=\mathcal{M}(\mathcal{E}[u])+ \int_{\mathbb{R}^N} \frac{1}{p(x)}V(x)|u|^{p(x)}dx \\
&\quad\quad- \int_{\mathbb{R}^N} \int_{\mathbb{R}^N} \frac{|u(x)|^{r(x)} |u(y)|^{r(y)}}{2r(x)|x-y|^{\alpha(x,y)}}dxdy-\int_{\mathbb{R}^N} \frac{1}{q_c(x)}|u|^{q_c(x)}dx \\
&\quad\quad -\varepsilon_W\int_{\mathbb{R}^N} \frac{1}{q_W(x)} W(x)|u|^{q_W(x)}dx.
\end{split}
\end{equation}
A critical point of $I$ is called a \textit{weak solution} or simply a \textit{solution} of \eqref{cho}.

\section{A result on Sobolev embeddings}
Next, we state the variable-order Sobolev embedding theorem with variable exponents. 
Combining the proofs of Theorem 3.2 in \cite{variable-order_p(x)_Sobolev} and Theorem 3.2 in \cite{uniformly_continuous} (or Theorem 3.6 in \cite{variable-order_p(x)}), we obtain the following compact embedding result.
\begin{proposition}
Let $\Omega\subset\mathbb{R}^N$ be a bounded Lipschitz domain. Assume that continuous and symmetric functions $p\in C_{+}(\bar{\Omega}\times \bar{\Omega})$ and $s: \bar{\Omega}\times \bar{\Omega} \to\mathbb{R}$ satisfy $0<s^{-}\leq s^{+}<1$ and $s^{+}p^{+}<N$ and that $\beta\in C_{+}(\bar{\Omega})$ satisfies $\beta (x)\geq p(x)$ for any $x\in \bar{\Omega}$. Let $q\in C_{+}(\bar{\Omega})$ be such that $\beta (x)\leq q(x)< p_s^*(x)$ for any $x\in \bar{\Omega}$ (and thus $\inf(p_s^*-q)>0$ since $\bar{\Omega}$ is compact). Then, there exists a constant $C>0$ such that for any $u\in W^{s(x,y),\beta(x),p(x,y)}(\Omega)$, we have
\[
\|u\|_{L^{q(x)}(\Omega)}\leq C \|u\|_{W^{s(x,y),\beta(x),p(x,y)}(\Omega)}.
\]
Moreover, this embedding is compact.
\end{proposition}
For critical embeddings, we need the log-H\"{o}lder continuity of exponents (see Corollary 8.3.2 and Proposition 8.3.7 in \cite{Diening}). As for the critical case, we refer to Theorem 3.3 in \cite{s_variable_exponent_CC} and its application \cite{Liang}. By Theorem 7 in \cite{variable_triebel_lizorkin}, we obtain the following Sobolev embedding result for variable exponent Triebel-Lizorkin spaces.
\begin{proposition}
Let $p_1,p_2$ and $q_1,q_2$ be bounded and globally log-H\"{o}lder continuous exponents with positive infima. Let $s_1,s_2$ be bounded and locally log-H\"{o}lder continuous functions with finite limits at infinity. If $s_2-N/p_2= s_1-N/p_1$, $p_2\geq p_1$ and $(s_1-s_2)^{-}>0$, then the continuous embedding $F_{p_1(\cdot),q_1(\cdot)}^{s_1(\cdot)}(\mathbb{R}^N)\hookrightarrow F_{p_2(\cdot),q_2(\cdot)}^{s_2(\cdot)}(\mathbb{R}^N)$ holds.
\end{proposition}

By Theorem 4.5 in \cite{100}, we have $F_{p(\cdot),2}^{s} (\mathbb{R}^N)\cong H^{s,p(\cdot)} (\mathbb{R}^N)$ ($s\in [0,\infty)$) and $H^{m,p(\cdot)} (\mathbb{R}^N)\cong W^{m,p(\cdot)} (\mathbb{R}^N)$ ($m\in\mathbb{N}$), where $H^{s,p(\cdot)} (\mathbb{R}^N)$ is a variable exponent Bessel potential space. From the continuous embedding for variable exponent Triebel-Lizorkin spaces, we conclude $W^{m,p(x)} (\mathbb{R}^N)\hookrightarrow L^{p_m^*(x)} (\mathbb{R}^N)$ for $m\in\mathbb{N}$. For $m=1$, also refer to Theorem 8.3.1 in \cite{Diening}. However, $W^{s,p} (\mathbb{R}^N)=F_{p,p}^{s}(\mathbb{R}^N)$ ($s\notin\mathbb{Z}$) is no longer generally valid in the variable exponent case, so the continuous embedding for variable exponent Triebel-Lizorkin spaces does not imply the critical Sobolev embedding for Slobodeckij space with variable smoothness and variable integrability. \par
Generalizing Theorem 3.3 in \cite{s_variable_exponent_CC}, we obtain the following critical embedding result. 
\begin{theorem}
Assume that symmetric, bounded and globally log-H\"{o}lder continuous functions $p\in C_{+}(\mathbb{R}^N\times \mathbb{R}^N)$ and $s: \mathbb{R}^N\times \mathbb{R}^N\to\mathbb{R}$ satisfy $0<s^{-}\leq s^{+}<1$ and $s^{+}p^{+}<N$. In addition, assume $p(x,y)$ and $s(x,y)$ is constant in $\delta$-neighborhood of $\{x=y\}$ for some $\delta>0$. Let $q\in C_{+}(\mathbb{R}^N)$ be a uniformly continuous function such that $p^{+}\leq q^{-}\leq q(x)\leq p_s^*(x)$. Then, the continuous embedding $W^{s(x,y),p(x,y)}(\mathbb{R}^N)\hookrightarrow L^{q(x)}(\mathbb{R}^N)$ holds.
\end{theorem}
\begin{proof}
First, we impose the following additional assumption on $p(x,y),s(x,y)$:
\begin{itemize}
\item[(E1)] For any $x\in\mathbb{R}^N$, there exists $\varepsilon=\varepsilon(x)>0$ such that
\[
q|_{B_\varepsilon(x)}^{+}\leq\frac{Np|_{B_\varepsilon(x)\times B_\varepsilon(x)}^{-}}{N-p|_{B_\varepsilon(x) \times B_\varepsilon(x)}^{-} s|_{B_\varepsilon(x) \times B_\varepsilon(x)}^{-}}.
\]
\end{itemize}
Let us cover $\mathbb{R}^N$ with a family $\{Q_i\}_{i\in\mathbb{N}}$ of countably many closed cubes with sides of length $\varepsilon<1$ so small that
\[
c_{p,\varepsilon}\coloneqq\sup_{0<|x-y|<1/2}|p(x,y)-p|_{B_\varepsilon(x)\times B_\varepsilon(y)}^{-}|\log\frac{1}{|x-y|}<\infty,
\]
\[
c_{s,\varepsilon}\coloneqq\sup_{0<|x-y|<1/2}|s(x,y)-s|_{B_\varepsilon(x)\times B_\varepsilon(y)}^{-}| \log\frac{1}{|x-y|} <\infty.
\]
By (E1) and uniform continuity, taking $\varepsilon>0$ small enough, we also have
\[
q_i^{+}\leq\frac{Np_i^{-}}{N-p_i^{-}s_i^{-}},
\]
where $q_i^{\pm}\coloneqq q|_{Q_i}^{\pm}$, $p_i^{\pm}\coloneqq p|_{Q_i}^{\pm}$, $s_i^{\pm}\coloneqq s|_{Q_i}^{\pm}$. Let $v\in W^{s(x,y),p(x,y)}(\mathbb{R}^N)$ and $i\in\mathbb{N}$. Define a measure $\mu_i$ and a function $F$ as follows:
\[
d\mu(x,y)\coloneqq\frac{dxdy}{|x-y|^{N-s_i^{-}p_i^{-}}},\quad F(x,y)\coloneqq \frac{|v(x)-v(y)|}{|x-y|^{2s_i^{-}}}.
\]
Since
\begin{align*}
&\phantom{\coloneqq} \sup_{(x,y):x\neq y}|x-y|^{-(s(x,y)-s_i^{-})p(x,y)-s_i^{-}(p(x,y)-p_i^{-})} \\
&= \sup_{(x,y):x\neq y} e^{-(s(x,y)-s_i^{-})p(x,y)\log|x-y|-s_i^{-}(p(x,y)-p_i^{-})\log|x-y|} \\
&\leq \sup_{(x,y):x\neq y} e^{-p^{+}(s(x,y)-s_i^{-})\log|x-y|-s^{+}(p(x,y)-p_i^{-})\log|x-y|} \eqqcolon C_L<\infty,
\end{align*}
by the definition of $[\cdot]_{s(x),p(x);Q_i}$, we have
\begin{align*}
1&=\int_{Q_i}\int_{Q_i}\frac{|v(x)-v(y)|^{p(x,y)}}{[v]_{s,p;Q_i}^{p(x,y)}|x-y|^{N+s(x,y)p(x,y)}}dxdy \\
&= \int_{Q_i}\int_{Q_i}\left|\frac{F(x,y)}{[v]_{s,p;Q_i}}\right|^{p(x,y)}\frac{1}{|x-y|^{-(s(x,y)-s_i^{-})p(x,y)-s_i^{-}(p(x,y)-p_i^{-})}} d\mu_i(x,y) \\
&\geq (C_L+1)^{-1} \int_{Q_i}\int_{Q_i}\left|\frac{F(x,y)}{[v]_{s,p;Q_i}}\right|^{p(x,y)}d\mu_i(x,y) \\
&\geq \int_{Q_i}\int_{Q_i}\left|\frac{F(x,y)}{(C_L+1)^{1/p_i^{-}} [v]_{s,p;Q_i}}\right|^{p(x,y)}d\mu_i(x,y).
\end{align*}
Therefore, $\|F\|_{L_{\mu_i}^{p(x,y)}(Q_i\times Q_i)}\leq (C_L+1)^{1/p_i^{-}} [v]_{s,p;Q_i}$. On the other hand,
\begin{align*}
\|F\|_{L_{\mu_i}^{p_i^{-}}(Q_i\times Q_i)}&= \left( \int_{Q_i}\int_{Q_i}\left|\frac{|v(x)-v(y)|}{|x-y|^{2s_i^{-}}}\right|^{p_i^{-}}\frac{dxdy}{|x-y|^{N-s_i^{-}p_i^{-}}} \right)^{1/p_i^{-}} \\
&= \left( \int_{Q_i}\int_{Q_i}\frac{|v(x)-v(y)|^{p_i^{-}}}{|x-y|^{N+s_i^{-}p_i^{-}}} dxdy \right)^{1/p_i^{-}}=[v]_{s_i^{-},p_i^{-};Q_i}
\end{align*}
and since $p_i^{-}\leq p|_{Q_i\times Q_i}$, by Corollary 3.3.4 in \cite{Diening},
\begin{align*}
\|F\|_{L_{\mu_i}^{p_i^{-}}(Q_i\times Q_i)}&\leq 2(1+\mu_i (Q_i\times Q_i)) \|F\|_{L_{\mu_i}^{p(x,y)}(Q_i\times Q_i)} \\
&\leq  2\left(1+|Q_i|\cdot\frac{N|B_1|\cdot 2^{s_i^{-}p_i^{-}}}{s_i^{-}p_i^{-}}\right) \|F\|_{L_{\mu_i}^{p(x,y)}(Q_i\times Q_i)}.
\end{align*}
Combining these, we obtain $[v]_{s_i^{-},p_i^{-};Q_i}\leq C [v]_{s(x),p(x);Q_i}$. Since $p_i^{-}\leq p|_{Q_i}$, we also have $\|v\|_{L^{p_i^{-}}(Q_i)}\leq 2(1+|Q_i|) \|v\|_{L^{p(x)}(Q_i)}$. By summing them, we get $\|v\|_{W^{s_i^{-},p_i^{-}}(Q_i)}\leq \|v\|_{W^{s(\cdot),p(\cdot)}(Q_i)}$. \par
As in the proof of Theorem 3.5 in \cite{uniformly_continuous}, we can deduce the existence of an extension $\tilde{v}\in W^{s_i^{-},p_i^{-}}(\mathbb{R}^N)$ with compact support such that $\tilde{v}=v$ in $Q_i$, and $\|v\|_{L^{q_i^{+}}(Q_i)}\leq\| \tilde{v}\|_{L^{(p_i^{-})_{s_i^{-}}^*}(\mathbb{R}^N)}\leq C\|v\|_{W^{s_i^{-},p_i^{-}}(Q_i)}$. On the other hand, $\|v\|_{L^{q(x)}(Q_i)}\leq 2(1+|Q_i|) \|v\|_{L^{q_i^{+}}(Q_i)}$. Combining these, we obtain $\|v\|_{L^{q(x)}(Q_i)}\leq C \|v\|_{W^{s(\cdot),p(\cdot)}(Q_i)}$. \par
In order to prove the embedding, by the closed graph theorem and the homogeneity of the inequality, it suffices to show that we have $\varrho_{L^{q(x)}(\mathbb{R}^N)}(v)<\infty$ for any $v\in W^{s(x,y),p(x,y)}(\mathbb{R}^N)$ such that $\|v\|_{W^{s(x,y),p(x,y)}}=1$. Take such $v$ arbitrarily. Then, $\varrho_{Q_i}(v)\leq \varrho_{W^{s(\cdot),p(\cdot)}}=1$ and thus $\|v\|_{W^{s(\cdot),p(\cdot)}(Q_i)}\leq 1$. 
\begin{align*}
\int_{Q_i}|v|^{q(x)}dx&\leq \max\{\|v\|_{L^{q(x)}(Q_i)}^{q_i^{+}}, \|v\|_{L^{q(x)}(Q_i)}^{q_i^{-}}\}\\
&\leq C\|v\|_{W^{s(\cdot),p(\cdot)}(Q_i)}^{q_i^{-}} \\
&\leq C\varrho_{W^{s(\cdot),p(\cdot)}(Q_i)}(v)^{q_i^{-}/p_i^{+}} \\
&\leq C\varrho_{W^{s(\cdot),p(\cdot)}(Q_i)}(v).
\end{align*}
Summing up for all $i$, we obtain 
\begin{equation}\label{E}
\varrho_{L^{q(x)}(\mathbb{R}^N)}(v)\leq C\varrho_{W^{s(\cdot),p(\cdot)}(\mathbb{R}^N)}(v) <\infty.
\end{equation}
Note that we can take $C$ in \eqref{E} depending only on $N$, $p^{-}$, $p^{+}$, $s^{-}$, $s^{+}$ and $C_L$. Such a constant does not explicitly depend on $\varepsilon$. $C_L$ depends only on $p^{-}$, $p^{+}$. \par
Next, we consider the general case without the condition (E1). Take any $v\in C_c^\infty(\mathbb{R}^N)$. Let $\{q_n\}$ be a sequence satisfying (E1) with $\varepsilon=\varepsilon_n=o_n(1)$ and approximating $q_c$ pointwise from below. By Fatou's lemma (more precisely by Corollary 3.5.4 in \cite{Diening}), $\displaystyle \|v\|_{q_c(x)}\leq\liminf_{n\to\infty}\|v\|_{q_n(x)}$. Since we can take $C$ in \eqref{E} independent of $\varepsilon$, we get $\|v\|_{q_c(x)}\leq C \|v\|_{W^{s(x,y),p(x,y)}}$. We can conclude by a density argument. 
\end{proof}

\section{A new concentration compactness lemma}
\begin{lemma} \label{CC}
Suppose (S), (P) and (QC). Let $\{u_n\}\subset W^{s(x,y),p(x,y)}(\mathbb{R}^N)$ be a bounded sequence converging to some $u$ weakly in $W^{s(x,y),p(x,y)}(\mathbb{R}^N)$ and almost everywhere. 
Assume
\begin{align*}
\displaystyle\int_{\mathbb{R}^N} \frac{|u_n(x)-u_n(y)|^{p(x,y)}}{|x-y|^{N+p(x,y)s(x,y)}}dy& \xrightharpoonup[]{\ast} \mu \\
|u_n|^{q_c (x)} & \xrightharpoonup[]{\ast} \nu \\
\int_{\mathbb{R}^N}\frac{|u_n(x)|^{r(x)}|u_n(y)|^{r(y)}}{|x-y|^{\alpha(x,y)}}dy & \xrightharpoonup[]{\ast} \xi
\end{align*}
in the sense of the vague convergence. 
Then, there exist an at most countable set $\mathcal{I}$, a family of points $\{x_i\}_{i\in \mathcal{I}}\subset\mathbb{R}^N$ and families of nonnegative numbers $\{\xi_i\}_{i\in \mathcal{I}}, \{\mu_i\}_{i\in \mathcal{I}}, \{\nu_i\}_{i\in \mathcal{I}}\subset \mathbb{R}$ such that
\begin{align}
\xi &= \int_{\mathbb{R}^N} \frac{|u(x)|^{r(x)}|u(y)|^{r(y)}}{|x-y|^{\alpha(x,y)}} dy+\sum_{i\in \mathcal{I}}\xi_i\delta_{x_i}; \label{xi} \\
\mu &\geq \int_{\mathbb{R}^N} \frac{|u(x)-u(y)|^{p(x,y)}}{|x-y|^{N+p(x,y)s(x,y)}} dy +\sum_{i\in \mathcal{I}}\mu_i \delta_{x_i}; \label{mu} \\
\nu &=|u|^{q_c(x)}+\sum_{i\in \mathcal{I}}\nu_i\delta_{x_i}. \label{nu}
\end{align}
Moreover, such numbers satisfy the following quantitative relations:
\begin{align}
\xi_i&\leq 2S_K^{-1}\max\{\mu_i^{2r^{+}/p^{+}}, \mu_i^{2r^{+}/p^{-}}, \mu_i^{2r^{-}/p^{+}}, \mu_i^{2r^{-}/p^{-}}\}; \label{SK quantity}\\
\xi_i&\leq 2L_K^{-1}\nu_i^{2/\sigma_\alpha^{+}}; \label{LK quantity}\\
\nu_i&\leq S_D^{-p_s^*(x_i)}\mu_i^{p_s^*(x_i)/p(x_i)}, \label{SD quantity}
\end{align}
where
\begin{align*}
S_K&\coloneqq \inf_{v\in W^{s(x,y),p(x,y)} (\mathbb{R}^N)\setminus\{0\}}\frac{\|v\|_{W^{s(x,y),p(x,y)}}^{2 r^{+}}+\|v\|_{W^{s(x,y),p(x,y)}}^{2 r^{-}}}{\int_{\mathbb{R}^N\times \mathbb{R}^N} \frac{|v(x)|^{r(x)}|v(y)|^{r(y)}}{|x-y|^{\alpha(x,y)}} dxdy}>0;\\
L_K&\coloneqq\inf_{v\in W^{s(x,y),p(x,y)} (\mathbb{R}^N)\setminus\{0\}}\frac{\||v|^{r(x)}\|_{\sigma^{+}_\alpha}^{2} +\||v|^{r(x)}\|_{\sigma^{-}_\alpha}^{2}}{\int_{\mathbb{R}^N\times \mathbb{R}^N} \frac{|v(x)|^{r(x)}|v(y)|^{r(y)}}{|x-y|^{\alpha(x,y)}} dxdy}>0;\\
S_D&\coloneqq \inf_{v\in W^{s(x,y),p(x,y)} (\mathbb{R}^N)\setminus\{0\}}\frac{[v]_{s(x,y),p(x,y)}}{\|v\|_{q_c(x)}}>0.
\end{align*}
Furthermore, we have
\begin{align}
\limsup_{n\to\infty}\int_{\mathbb{R}^N} \int_{\mathbb{R}^N} \frac{|u_n(x)|^{r(x)}|u_n(y)|^{r(y)}}{|x-y|^{\alpha(x,y)}}dxdy&=\xi(\mathbb{R}^N)+\xi_\infty;\label{mass conservation for xi}\\
\limsup_{n\to\infty} \int_{\mathbb{R}^N} \int_{\mathbb{R}^N} \frac{|u_n(x)-u_n(y)|^{p(x,y)}}{|x-y|^{N+p(x,y)s(x,y)}}dxdy&=\mu(\mathbb{R}^N)+\mu_\infty; \label{mass conservation for mu}\\
\limsup_{n\to\infty}\int_{\mathbb{R}^N} |u_n|^{q_c(x)} dx &= \nu(\mathbb{R}^N)+\nu_\infty, \label{mass conservation for nu}
\end{align}
where \textit{the escaping parts} $\mu_\infty,\nu_\infty,\xi_\infty$ are defined as
\begin{align}
\mu_\infty&\coloneqq\lim_{R\to\infty}\limsup_{n\to\infty}\int_{\{|x|\geq R\}} \left(\int_{\mathbb{R}^N}\frac{|u_n(x)-u_n(y)|^{p(x,y)}}{|x-y|^{N+p(x,y)s(x,y)}}dy\right)dx; \\
\nu_\infty&\coloneqq \lim_{R\to\infty}\limsup_{n\to\infty}\int_{\{|x|\geq R\}} |u_n|^{q_c(x)} dx; \\
\xi_\infty&\coloneqq \lim_{R\to\infty}\limsup_{n\to\infty}\int_{\{|x|\geq R\}} \frac{|u_n(x)|^{r(x)}|u_n(y)|^{r(y)}}{|x-y|^{\alpha(x,y)}} dx.
\end{align}
In addition, if $\displaystyle p_\infty\coloneqq \lim_{|x|,|y|\to\infty}p(x,y)$ and $\displaystyle (q_c)_\infty \coloneqq \lim_{|x|\to\infty}q_c(x)$ exist, 
\begin{align}
S_{D}^{(q_c)_\infty}\nu_\infty&\leq\mu_\infty^{(q_c)_\infty/p_\infty}. \label{mugeqnuatinfty}
\end{align}
\end{lemma}
\begin{proof}
Let $v_n\coloneqq u_n-u$. By the variable-exponent Brezis-Lieb type splitting properties (whose proof is omitted since it is obtained in the same way as in the case of constant exponents), we obtain
\begin{align*}
\int_{\mathbb{R}^N}\frac{|v_n(x)-v_n(y)|^{p(x,y)}}{|x-y|^{N+p(x,y)s(x,y)}}dy &\to \tilde{\mu}\coloneqq \mu-\int_{\mathbb{R}^N}\frac{|u(x)-u(y)|^{p(x,y)}}{|x-y|^{N+p(x,y)s(x,y)}}dy \\
|v_n|^{q_c(x)} &\to \tilde{\nu}\coloneqq\nu -|u|^{q_c(x)} \\
\
\int_{\mathbb{R}^N}\frac{|v_n(x)|^{r(x)}|v_n(y)|^{r(y)}}{|x-y|^{\alpha(x,y)}}dy &\to \xi-\int_{\mathbb{R}^N}\frac{|u(x)|^{r(x)}|u(y)|^{r(y)}}{|x-y|^{\alpha(x,y)}}dy
\end{align*}
in the vague topology. \par
By a similar way as the proof of Theorem 3.3 in \cite{Ho_Sim}, Theorem 3.3 in \cite{variable_exponent_CC} or Theorem 2.2 in \cite{Fu_Zhang}, we obtain the existence of at most countably many distinct points $\{x_i\}_{i\in\mathcal{I}'}\subset \mathbb{R}^N$ and $\{\mu_i\}_{i\in\mathcal{I}'}, \{\nu_i\}_{i\in\mathcal{I}'}\subset (0,\infty)$ such that \eqref{mu}, \eqref{nu} and \eqref{SD quantity} holds. \par
Take any $\varphi\in C_c^\infty(\mathbb{R}^N)$. We observe 
\begin{align*}
&\phantom{=}\left|\int_{\mathbb{R}^N} \int_{\mathbb{R}^N}\frac{|\varphi(x)v_n(x)|^{r(x)}|\varphi(y)v_n(y)|^{r(y)}}{|x-y|^{\alpha(x,y)}}dxdy \right.\\
&\phantom{=}\quad\quad\left. -\int_{\mathbb{R}^N} |\varphi(x)|^{2 r(x)}\int_{\mathbb{R}^N}\frac{|v_n(x)|^{r(x)}|v_n(y)|^{r(y)}}{|x-y|^{\alpha(x,y)}}dydx\right| \\
&=\left| \int_{\mathbb{R}^N}\Gamma_n (x)dx\right|= \left| \int_{\operatorname{supp}\varphi}\Gamma_n (x)dx\right|,
\end{align*}
where
\begin{align*}
\Gamma_n(x)&\coloneqq \int_{\mathbb{R}^N}\frac{|\varphi(y)|^{r(y)}-|\varphi(x)|^{r(x)}}{|x-y|^{\alpha(x,y)}} |v_n(y)|^{r(y)} dy\cdot |\varphi(x)v_n(x)|^{r(x)} \\
&=\left(\int_{|y|\leq R}\frac{|\varphi(y)|^{r(y)}-|\varphi(x)|^{r(x)}}{|x-y|^{\alpha(x,y)}} |v_n(y)|^{r(y)} dy \right.\\
&\phantom{=}\quad\quad\left.-|\varphi(x)|^{r(x)}\int_{|y|> R}\frac{|v_n(y)|^{r(y)}}{|x-y|^{\alpha(x,y)}} dy\right) |\varphi(x)v_n(x)|^{r(x)}
\end{align*}
for $R>0$ sufficiently large. 
Since $x\mapsto |\varphi(x)|^{r(x)}$ is differentiable and has a compact support, $\displaystyle (x,y)\mapsto \frac{|\varphi(y)|^{r(y)}-|\varphi(x)|^{r(x)}}{|x-y|}$ is bounded. By the same technique as Proposition 2.4 in \cite{variable_exponent_HLS} and the Young's inequality for the weak Lebesgue space, we get
\begin{align*}
&\phantom{=}\left\|\int_{|y|\leq R}\frac{|v_n(y)|^{r(y)}}{|x-y|^{\alpha(x,y)-1}} dy\right\|_{L^{\theta}(\operatorname{supp}\varphi)} \\
&\leq \left\|\int_{|y|\leq R}\frac{|v_n(y)|^{r(y)}}{|x-y|^{\alpha^{-}-1}} dy\right\|_{L^{\theta}(\operatorname{supp}\varphi)} +\left\|\int_{|y|\leq R}\frac{|v_n(y)|^{r(y)}}{|x-y|^{\alpha^{+}-1}} dy\right\|_{L^{\theta}(\operatorname{supp}\varphi)}\\
&\leq C\||\cdot|^{1-\alpha^{-}}\|_{L^{q_\alpha} (\operatorname{supp}\varphi-B_R(0))}\||v_n|^{r(x)}\|_{\sigma_\alpha^{+}}\leq C'
\end{align*}
for $\theta\geq 1$ sufficiently large such that
\[
1+\frac{1}{\theta}=\frac{1}{q_\alpha}+\frac{1}{\sigma_\alpha^{+}}
\]
for some $q_\alpha$ such that
\[
1\leq q_\alpha<\min\left\{\frac{N}{\max\{\alpha^{+}-1,0\}}, \frac{N}{\max\{\alpha^{-}-1,0\}}\right\},
\]
where we interpret as $N/0=\infty$. On the other hand, 
\begin{align*}
&\phantom{=}\left\| |\varphi(x)|^{r(x)}\int_{|y|> R}\frac{|v_n(y)|^{r(y)}}{|x-y|^{\alpha(x,y)}} dy\right\|_{L^{\theta}(\operatorname{supp}\varphi)} \\
&\leq \left\| |\cdot|^{-\alpha^{-}}\wedge\sup_{x\in \operatorname{supp}\varphi, |y|>R}|x-y|^{-\alpha(x,y)}\right\|_{q_K}\||v_n|^{r(x)}\|_{\sigma_\alpha^{+}}\||\varphi|^{r(x)}\|_{q_\varphi}\leq C
\end{align*}
for $q_K\in (N/\alpha^{-},\infty]$ and $q_\varphi\geq 1$ such that
\[
1+\frac{1}{\theta}=\frac{1}{q_K}+\frac{1}{\sigma_\alpha^{+}}+ \frac{1}{q_\varphi}.
\]
Combining these, by the H\"{o}lder's inequality, we obtain
\begin{align*}
&\phantom{=}\|\Gamma_n\|_{L^{q_\Gamma}(\operatorname{supp}\varphi)} \\
&\leq \left\| \int_{\mathbb{R}^N}\frac{|\varphi(y)|^{r(y)}-|\varphi(x)|^{r(x)}}{|x-y|^{\alpha(x,y)}} |v_n(y)|^{r(y)} dy\right\|_{L^{\theta}(\operatorname{supp}\varphi)}\||v_n|^{r(x)}\|_{\sigma_\alpha^{+}}\leq C
\end{align*}
for $q_\Gamma>1$ such that $\displaystyle \frac{1}{q_\Gamma}=\frac{1}{\theta}+\frac{1}{\sigma_\alpha^{+}}$. Therefore, by the Vitali convergence theorem and the fact that $\Gamma_n\to 0$ a.e. in $\mathbb{R}^N$, we obtain $\displaystyle \int_{\operatorname{supp}\varphi}|\Gamma_n|dx\to 0$.\par
By Proposition \ref{HLS} and the definition of $L_K$, this implies
\begin{align*}
&\phantom{=}L_K\int_{\mathbb{R}^N}|\varphi|^{2r(x)}\left(\int_{\mathbb{R}^N}\frac{|v_n(x)|^{r(x)}|v_n(y)|^{r(y)}}{|x-y|^{\alpha(x,y)}}dy\right)dx \\
&\leq \||\varphi v_n|^{r(x)}\|_{\sigma_\alpha^{+}}^2+ \||\varphi v_n|^{r(x)}\|_{\sigma_\alpha^{-}}^2+o_n(1)
\end{align*}
as $n\to\infty$. From this,
\begin{equation}\label{pre LK quantity}
\begin{split}
&\phantom{=} \min_{\pm}\| \varphi \|_{L^{2 r(x)}_{\tilde{\xi}}(\operatorname{supp}\varphi)}^{2 r|_{\operatorname{supp}\varphi}^{\pm}} \leq \| |\varphi|^{2r(x)} \|_{L^1_{\tilde{\xi}}(\operatorname{supp}\varphi)} \\
& \leq L_K^{-1}\left(\max_{\pm}\|\varphi v_n\|_{L^{r(x)\sigma_\alpha^{+}}(\operatorname{supp}\varphi)}^{2 r|_{\operatorname{supp}\varphi}^{\pm}}+ \max_{\pm}\|\varphi v_n\|_{L^{r(x)\sigma_\alpha^{-}}(\operatorname{supp}\varphi)}^{2 r|_{\operatorname{supp}\varphi}^{\pm}}\right)+o_n(1) \\
&= L_K^{-1}\left(\max_{\pm}\|\varphi\|_{L^{r(x)\sigma_\alpha^{+}}_{|v_n|^{r(x)\sigma_\alpha^{+}}}(\operatorname{supp}\varphi)}^{2 r|_{\operatorname{supp}\varphi}^{\pm}}+ \max_{\pm}\|\varphi\|_{L^{r(x)\sigma_\alpha^{-}}_{|v_n|^{r(x)\sigma_\alpha^{-}}}(\operatorname{supp}\varphi)}^{2 r|_{\operatorname{supp}\varphi}^{\pm}}\right)+o_n(1),
\end{split}
\end{equation}
where $\max_{\pm}A_\pm$ [resp. $\min_{\pm}A_\pm$] denotes $\max\{A_{+},A_{-}\}$ [resp. $\min\{A_{+},A_{-}\}$]. 
Arguing as the proof of Lemma 3.7 and Lemma 3.8 in \cite{variable_exponent_CC}, $\tilde{\xi}$ is also a linear combination $\sum_{i\in \mathcal{I}''}\xi_i\delta_{x_i}$ of at most countably many Dirac measures. We set $\mathcal{I}\coloneqq \mathcal{I}'\cup \mathcal{I}''$. 
For simplicity, we consider the case $\sigma_\alpha^{-}<\sigma_\alpha^{+}$. Let $\mathcal{A}\coloneqq \{x\in\mathbb{R}^N\mid r(x)\sigma_\alpha^{+}=p_s^*(x)\}$. If $\operatorname{supp}\varphi \subset \mathbb{R}^N\setminus \mathcal{A}$, by the local compact embedding, $\|\varphi v_n\|_{L^{r(x)\sigma_\alpha^{\pm}}(\operatorname{supp}\varphi)}\to 0$. Therefore, we may assume $\{x_i\}_{i\in\mathcal{I}''}\subset \mathcal{A}$. Take $\phi\in C_c^\infty(\mathbb{R}^N)$ such that $\chi_{B_1(0)}\leq \phi\leq \chi_{B_2(0)}$ and define $\phi_{a,\varepsilon}(x)\coloneqq\phi((x-a)/\varepsilon)$. By the Banach–Alaoglu theorem for the space of finite Radon measures, $\{|v_n|^{r(x) \sigma_\alpha^{+}}\}$ has a limit $\mathcal{V}$ up to a subsequence in the vague topology. For $x_0\in\mathcal{A}$, we have
\[
\lim_{\varepsilon\to +0}\lim_{n\to\infty}\int_{\mathbb{R}^N} | \phi_{x_0,\varepsilon} v_n|^{r(x) \sigma_\alpha^{+}}dx= \lim_{\varepsilon\to +0}\int_{\mathbb{R}^N} | \phi_{x_0,\varepsilon}|^{r(x) \sigma_\alpha^{+}}d\mathcal{V}=\mathcal{V}(x_0)=\tilde{\nu}(\{x_0\}).
\]
Choosing $\varphi=\phi_{x_i,\varepsilon}$ in \eqref{pre LK quantity} and taking the limit as $\varepsilon\to +0$ after $n\to\infty$, by the relation between the Luxemburg norm and the modular, we obtain
\[
\xi_i\leq 2L_K^{-1}(\nu_i^{\frac{1}{r(x_i)\sigma_{+}}})^{2r(x_i)}= 2L_K^{-1}\nu_i^{2/\sigma_{+}},
\]
that is, \eqref{LK quantity} holds. Similarly, noting that $\|\varphi v_n\|_{p(x)}\to 0$ by the local compact embedding, in the same way as in \cite{Sakuma}, we obtain
\begin{align*}
\xi_i&\leq S_K^{-1}(\max_{\pm}\mu_i^{2r^{+}/p^{\pm}}+ \max_{\pm}\mu_i^{2r^{-}/p^{\pm}}) \\
&\leq 2S_K^{-1}\max\{\mu_i^{2r^{+}/p^{+}}, \mu_i^{2r^{+}/p^{-}}, \mu_i^{2r^{-}/p^{+}}, \mu_i^{2r^{-}/p^{-}}\},
\end{align*}
that is, \eqref{SK quantity} holds. 
In order to show \eqref{mass conservation for xi}, \eqref{mass conservation for mu}, \eqref{mass conservation for nu}, it suffices to take the limit as $R\to\infty$ in
\begin{align*}
\limsup_{n\to\infty}\int_{\mathbb{R}^N}f_n dx &= \limsup_{n\to\infty}\int_{\mathbb{R}^N}f_n (x)(1-\phi(x/R))dx+ \lim_{n\to\infty}\int_{\mathbb{R}^N} \phi(x/R) (f_n dx)
\end{align*}
for each integrand $f_n$. We can also easily obtain \eqref{mugeqnuatinfty} by the definition of $S_D$ and considering $v=(1-\phi(x/R))u_n$. 
\end{proof}

\section{$(PS)_c$ condition}
\begin{lemma}\label{boundedness}
If $\{u_n\}$ is a $(PS)$ sequence for $I$, then $\{u_n\}$ is bounded uniformly in $\varepsilon_W\in (0,\varepsilon)$ for any $\varepsilon>0$. 
\end{lemma}
\begin{proof}
Since there exists a constant $\bar{c}$ such that $I[u_n]\leq \bar{c}$ and $I'[u_n]\to 0$, for $n$ sufficiently large and for $\beta\in (1/(2r^{-}),1/\theta_M)$, by the condition (M) and the H\"{o}lder type inequality, 
\begin{equation}\label{bounded}
\begin{split}
&\phantom{=}\bar{c}+o(\|u_n\|)\geq I[u_n]-\beta I'[u_n]u_n \\
&\geq \left(\frac{1}{\theta_M}-\beta\right)M(\mathcal{E}[u_n])\min\{[u_n]_{s(x,y),p(x,y)}^{p^{+}}, [u_n]_{s(x,y),p(x,y)}^{p^{-}}\}+\left(\frac{1}{p^{+}}-\beta\right)\int_{\mathbb{R}^N}V|u_n|^{p(x)}dx \\
&\phantom{=}\quad +\left(\beta-\frac{1}{2r^{-}}\right)K'[u_n]u_n+\left(\beta-\frac{1}{(q_c)^{-}}\right)\int_{\mathbb{R}^N}|u_n|^{q_c(x)}dx \\
&\phantom{=}\quad -\left(\frac{1}{q_W^{-}}-\beta\right) \varepsilon_W\int_{\mathbb{R}^N}W|u_n|^{q_W(x)}dx \\
&\geq C_1\min\{\|u_n\|^{p^{+}}, \|u_n\|^{p^{-}}\}-C_2\varepsilon_W\max\{\|u_n\|^{q_W^{+}}, \|u_n\|^{q_W^{-}}\}.
\end{split}
\end{equation}
Since $q_W^{+}<p^{-}$, this implies that $\{u_n\}$ is bounded (uniformly in $\varepsilon_W$). 
\end{proof}

\begin{lemma}\label{weaktoweakcontinuity}
If $u_n \rightharpoonup u$ in $E$, then $K'[u_n]\rightharpoonup K'[u]$, $W|u_n|^{q_W(x)-2}u_n \rightharpoonup W|u|^{q_W(x)-2}u$, $|u_n|^{q_c(x)-2}u_n\rightharpoonup |u|^{q_c(x)-2}u$ in $E$. 
\end{lemma}
\begin{proof}
By the local compact embedding, passing to a subsequence, we have $u_n\to u$ a.e. 
In the same way as in \cite{variable_exponent_HLS}, we have
\[
\frac{2}{\sigma_\alpha^{\pm}}+\frac{\alpha^{\pm}}{N}=2
\]
and
\[
\frac{1}{|x-y|^{\alpha(x,y)}}\leq \frac{1}{|x-y|^{\alpha^{+}}}+ \frac{1}{|x-y|^{\alpha^{-}}}.
\]
By the Young's convolution inequality for the weak Lebesgue space, 
\[
\left\| \int_{\mathbb{R}^N}\frac{|u_n(x)|^{r(x)}}{|x-y|^{\alpha(x,y)}}dy\right\|_{L^{2N/\alpha^{+}}+ L^{2N/\alpha^{-}}}\leq C\sup_n(\| |u_n|^{r(x)}\|_{\sigma_\alpha^{+}}+ \| |u_n|^{r(x)}\|_{\sigma_\alpha^{-}})<\infty.
\]
Take any $\varphi\in C_c^\infty (\mathbb{R}^N)$ and $\varepsilon>0$ appropriately small. The boundedness in $L^{2N/\alpha^{+}}+ L^{2N/\alpha^{-}}$ and the almost everywhere convergence implies the weak convergence in $L^{2N/\alpha^{+}}(\operatorname{supp}\varphi)$. By the local compact embedding, up to a subsequence, $|u_n|^{r(x)-1}\to |u|^{r(x)-1}$ in $L^{\frac{r^{+}}{r^{+}-1}\sigma_\alpha^{+}(1-\varepsilon)}(\operatorname{supp}\varphi)$. By the H\"{o}lder's inequality, $|u_n|^{r(x)-2}u_n\varphi\to |u_n|^{r(x)-2}u_n\varphi$ in $L^{\sigma_\alpha^{+}}(\operatorname{supp}\varphi)$. Since $\sigma_\alpha^{+} = \sigma_{\alpha^{+}}=(2N/\alpha^{+})'$, $K'[u_n]\varphi\to K'[u]\varphi$ in $L^1$. By a density argument, we obtain $K'[u_n]\rightharpoonup K'[u]$. This is a result of thinking in terms of ``up to a subsequence'', but since any subsequence of $\{K'[u_n]\}$ has a subsequence that converges to the same limit $K'[u]$ similarly, it converges to this limit without passing to a subsequence consequently. The same also applies below. \par
As for the term including $W$, note that
\[
\frac{1}{\frac{q_c(x)}{q_c(x)-q_W(x)}}+\frac{1}{\frac{q_c(x)}{q_W(x)-1}}+\frac{1}{q_c(x)}=1.
\]
By the H\"{o}lder type inequality, $\{W|u_n|^{q_W(x)-2}u_n\}$ is bounded in $L^{\frac{q_c(x)}{q_c(x)-1}}(\mathbb{R}^N)$. Hence, it has a weakly convergent subsequence. By the almost convergence of $\{u_n\}$, its limit is nothing but $W|u|^{q_W(x)-2}u$. Similarly, $|u_n|^{q_c(x)-2}u_n\rightharpoonup |u|^{q_c(x)-2}u$. 
\end{proof}
\begin{lemma}\label{PS}
There exists $\varepsilon_0>0$ such that if $\varepsilon_W\in (0,\varepsilon_0)$, then $I$ satisfies $(PS)_c$ condition for $c<0$.
\end{lemma}
\begin{proof}
Let $\{u_n\}$ be an arbitrary $(PS)_c$ sequence for $I$. By Lemma \ref{boundedness}, $\{u_n\}$ is bounded in $W^{s(x,y),p(x,y)}_V$. By the reflexivity of $W^{s(x,y),p(x,y)}_V$ or $W^{s(x,y),p(x,y)}$ and the Banach–Alaoglu theorem for the space consisting of all finite Radon measures, passing to a subsequence, we may assume that the assumption of Lemma \ref{CC} holds. Take $\phi\in C_c^\infty(\mathbb{R}^N)$ such that $\chi_{B_1(0)}\leq \phi\leq \chi_{B_2(0)}$ and define $\phi_{a,\varepsilon}(x)\coloneqq\phi((x-a)/\varepsilon)$. 
For each $i\in\mathcal{I}$ and $\varepsilon>0$, since $\{\phi_{x_i,\varepsilon}u_n\}$ is bounded, we have
\[
\lim_{n\to\infty}I'[u_n](\phi_{x_i,\varepsilon}u_n)=0.
\]
Taking $\varepsilon\to +0$, we get $0\geq \mu_i-\xi_i-\nu_i$. By \eqref{LK quantity} and \eqref{SD quantity} in Lemma \ref{CC}, $\nu_i+\nu_i^{2/\sigma_\alpha^{+}}\geq C\nu_i^{p(x_i)/p_s^*(x_i)}\geq C\min\{\nu_i^{1-\frac{p^{-}s^{-}}{N}}, \nu_i^{1-\frac{p^{+}s^{+}}{N}}\}$ and thus there exists a constant $\Lambda_{\text{loc}}$ depending only on embedding constants such that $\nu_i=0$ or $\nu_i\geq \Lambda_{\text{loc}}$ ($\forall i\in\mathcal{I}$). In particular, $\mathcal{I}$ is a finite set since $\nu$ and $\xi$ are finite measures. \par
Next, we consider the escaping parts. Letting $\eta_R(x)\coloneqq 1-\phi(x/R)$, from $\displaystyle\lim_{R\to\infty}\lim_{n\to\infty}I'[u_n](\eta_R u_n)=0$, we get
\begin{equation}\label{infty quantity}
\mu_\infty+V_\infty=\xi_\infty+\nu_\infty,
\end{equation}
where
\[
V_\infty\coloneqq \lim_{R\to\infty}\limsup_{n\to\infty}\int_{\{|x|>R\}}V|u_n|^{p(x)}dx.
\]
Noting that for any $\varepsilon>0$, there exists $R(\varepsilon)>0$ such that $|r(x)-r_\infty|<\varepsilon$ ($|x|>R(\varepsilon)$), we get
\begin{align*}
L_K \xi_\infty & \leq C \lim_{R\to\infty}\limsup_{n\to\infty}( \|u_n^{r(x)}\|_{\sigma_\alpha^{+}} \|u_n^{r(x)}\eta_R\|_{\sigma_\alpha^{+}}+ \|u_n^{r(x)}\|_{\sigma_\alpha^{-}} \|u_n^{r(x)}\eta_R\|_{\sigma_\alpha^{-}}) \\
&\leq C' \lim_{R\to\infty}\limsup_{n\to\infty}\max_{\pm_1,\pm_2}\|u_n\|_{r(x)\sigma_\alpha^{\pm_1}}^{r^{\pm_2}} \max_{\pm}\|u_n^{r(x)}\eta_R\|_{\sigma_\alpha^{\pm}}\\
&\leq C'\max_{\pm_1,\pm_2}\sup_n\|u_n\|_{r(x)\sigma_\alpha^{\pm_1}}^{r^{\pm_2}}  \lim_{R\to\infty}\limsup_{n\to\infty} \max_{\pm}\|u_n^{p\cdot\frac{r\theta_{\pm}}{p}}u_n^{p_s^*\cdot\frac{r(1-\theta_{\pm})}{p_s^*}}\eta_R\|_{\sigma_\alpha^{\pm}}\\
&\leq C''\max_{\pm_1,\pm_2}\sup_n\|u_n\|_{r(x)\sigma_\alpha^{\pm_1}}^{r^{\pm_2}}  \\
&\phantom{=}\quad\times\lim_{R\to\infty}\limsup_{n\to\infty} \max_{\pm_1,\pm_2,\pm_3}\|u_n\eta_R\|_{p(x)}^{(r\theta_{\pm_1})|_{B_R (0)^c}^{\pm_2}} \|u_n\eta_R\|_{q_c(x)}^{(r\cdot(1-\theta_{\pm_1}))|_{B_R (0)^c}^{\pm_3}}\\
&= C''\max_{\pm_1,\pm_2}\sup_n\|u_n\|_{r(x)\sigma_\alpha^{\pm_1}}^{r^{\pm_2}}  \\
&\phantom{=}\quad\times\lim_{R\to\infty}\limsup_{n\to\infty} \max_{\pm}\|u_n\eta_R\|_{p(x)}^{r_\infty\theta_{\pm,\infty}} \|u_n\eta_R\|_{q_c(x)}^{r_\infty(1-\theta_{\pm,\infty})}\\
&\leq C'''\lim_{R\to\infty}\limsup_{n\to\infty} \max_{\pm}\left(\frac{V_\infty}{\tau_0}\right)^{r_\infty\theta_{\pm,\infty}/p_\infty} \nu_\infty^{r_\infty(1-\theta_{\pm,\infty})/(q_c)_\infty}
\end{align*}
for $\theta_{\pm}\in C(\mathbb{R}^N;[0,1])$ such that 
\[
\frac{1}{r(x)\sigma_\alpha^{\pm}}=\frac{\theta_{\pm}(x)}{p(x)}+\frac{1-\theta_{\pm}(x)}{q_c(x)}
\]
and $\theta_{\pm,\infty}\in [0,1]$ such that 
\[
\frac{1}{r_\infty\sigma_\alpha^{\pm}}=\frac{\theta_{\pm,\infty}}{p_\infty}+\frac{1-\theta_{\pm,\infty}}{(q_c)_\infty},
\]
where $\displaystyle\max_{\pm_1,\pm_2}A_{\pm_1,\pm_2}$ denotes $\max\{A_{+,+},A_{+,-},A_{-,+},A_{-,-}\}$ and we also define $\displaystyle\max_{\pm_1,\pm_2,\pm_3}A_{\pm_1,\pm_2,\pm_3}$ analogously for notational simplicity. Here, the constant $C'''$ depends only on given functions, given exponents and embedding constants. Combing this inequality and \eqref{infty quantity}, by the Young's inequality for products of numbers, for $\varepsilon>0$ small enough, we obtain
\begin{equation}\label{infty quantity 2}
\mu_\infty\leq \mu_\infty+(1-\varepsilon)V_\infty\leq C_\varepsilon\sum_{\pm}\nu_\infty^{\frac{r_\infty(1-\theta_{\pm,\infty})}{(q_c)_\infty}\cdot \frac{p_\infty}{p_\infty-\theta_{\pm,\infty}r_\infty}}+\nu_\infty.
\end{equation}
Note that since a function
\[
\Xi_{\pm} (r_\infty)\coloneqq\frac{r_\infty (1-\theta_{\pm,\infty})}{(q_c)_\infty}\cdot \frac{p_\infty}{p_\infty-\theta_{\pm,\infty}r_\infty}= \frac{r_\infty\sigma_\alpha^{\pm}-p_\infty}{r_\infty \sigma_\alpha^{\pm}-p_\infty \sigma_\alpha^{\pm} +(q_c)_\infty (\sigma_\alpha^{\pm}-1)}
\]
is monotonically increasing with respect to $r_\infty$ (considered to be a variable temporarily) and $\Xi_\pm (p_\infty)=p_\infty/(q_c)_\infty$, we always have $\Xi_\pm (r_\infty)> p_\infty/(q_c)_\infty$ ($\because r_\infty>p_\infty$). This, together with \eqref{infty quantity} and \eqref{mugeqnuatinfty}, implies that there exists $\Lambda_\infty>0$ depending only on embedding constants such that $\nu_\infty=0$ or $\nu_\infty >\Lambda_\infty$. \par
From \eqref{bounded} and $c<0$, we have $\displaystyle \sup_n\|u_n\|\leq 
C \max\{\varepsilon_W^{\frac{1}{p^{+}-q_W^{-}}}, \varepsilon_W^{\frac{1}{p^{-}-q_W^{+}}}\}$. On the other hand, for $n$ sufficiently large,
\begin{equation}\label{smallness}
\begin{split}
0 &>c+o_n(\|u_n\|) \\
&=I[u_n]-\beta I'[u_n]u_n \\
&\geq C_1 (\mathcal{E}[u_n]+K'[u_n]u_n +\||u_n|^{q_c(x)}\|_1)-C_2\varepsilon_W\|W|u_n|^{q_W(x)}\|_1 \\
&\geq C_1\left(\min_{i\in\mathcal{I}}\nu_i+\nu_\infty\right) -C_3\varepsilon_W\max\{\|u_n\|^{q_W^{+}}, \|u_n\|^{q_W^{-}}\}+o_n (1) \\
&\geq C_1\left(\min_{i\in\mathcal{I}}\nu_i+\nu_\infty\right) -C_4 \max\{\varepsilon_W^{1+\frac{q_W^{+}}{p^{+}-q_W^{-}}}, \varepsilon_W^{1+\frac{q_W^{-}}{p^{-}-q_W^{+}}}\}+o_n (1).
\end{split}
\end{equation}
Let $\varepsilon_W<1$. Suppose $\nu_i\geq \Lambda_{\text{loc}}$ ($\exists i\in\mathcal{I}$) or $\nu_\infty\geq \Lambda_{\infty}$. Then, $\displaystyle\Lambda_{\text{loc}}\leq C_5 \varepsilon_W^{1+\frac{q_W^{-}}{p^{-}-q_W^{+}}}$ or $\displaystyle\Lambda_{\infty}\leq C_5 \varepsilon_W^{1+\frac{q_W^{-}}{p^{-}-q_W^{+}}}$. If $\varepsilon_W$ is sufficiently small, we reach a contradiction. In the following, we consider such $\varepsilon_W$. Then, $\nu_i=0$ ($\forall i\in\mathcal{I}$) and $\nu_\infty=0$. Using $\displaystyle \min_{i\in\mathcal{I}}\mu_i$ and $\mu_\infty$ instead of $\displaystyle \min_{i\in\mathcal{I}}\nu_i$ and $\nu_\infty$ in \eqref{smallness} and noting that \eqref{SD quantity} and \eqref{mugeqnuatinfty} hold, we get $\mu_i=0$ ($\forall i\in\mathcal{I}$) and $\mu_\infty=0$ similarly. By \eqref{LK quantity} in Lemma \ref{CC}, we have $\xi_i=0$ ($\forall i\in\mathcal{I}$). From $\mu_\infty=0$, taking a limit in the Hardy-Littlewood-Sobolev type inequality, we have $\xi_\infty=0$. By \eqref{infty quantity}, we also obtain $V_\infty=0$. \par
Now, from $V_\infty=0$, we know that $\{V|u_n|^{p(x)}\}$ is tight. Moreover, since $\displaystyle\sup_{n}\||u_n|^{p(x)}\|_{q_c(x)/p(x)}<\infty$, for any bounded domain $\Omega$, we have
\[
\int_{\Omega}V |u_n|^{p(x)}dx\leq C\|V\|_{L^{\frac{q_c(x)}{q_c(x)-p(x)}}(\Omega)}.
\]
Therefore, $\{V|u_n|^{p(x)}\}$ has uniformly absolutely continuous integrals. By the Vitali convergence theorem, $\|V|u_n|^{p(x)}\|_1\to \|V|u|^{p(x)}\|_1$. In addition, from $\mu_i=\mu_\infty=0$, we have $\mathcal{E}'[u_n]u_n\to \mathcal{E}'[u]u$. By the uniform convexity of $E$ (or the Brezis-Lieb type splitting property) and by the equivalence of norm convergence and modular convergence, we deduce $u_n\to u$ in $E$.
\end{proof}

\begin{remark}
There is another proof of the last part: From $\mu_i=\mu_\infty=0$, $\nu_i=\nu_\infty=0$ and $\xi_i=\xi_\infty=0$, we have $\mathcal{E}'[u_n]u_n\to \mathcal{E}'[u]u$, $\||u_n|^{q_c}\|_1\to \||u|^{q_c}\|_1$ and $K'[u_n]u_n\to K'[u]u$. In addition, by the Vitali convergence theorem, 
\[
\int_{\mathbb{R}^N}W|u_n|^{q_W(x)}dx\to \int_{\mathbb{R}^N}W|u|^{q_W(x)}dx.
\]
In fact, since for any domain $\Omega$, we have
\begin{align*}
\sup_n\int_\Omega W|u_n|^{q_W(x)}dx\leq C(\sup_n \|u_n\|_{q_c(x)})\max_\pm\left\|W^{\frac{q_c(x)}{q_c(x)-q_W(x)}}\right\|_{L^1(\Omega)}^{1-\frac{q_W^\pm}{(q_c)^mp}},
\end{align*}
$\{W|u_n|^{q_W(x)}\}$ is uniformly integrable. Since modulars $v\mapsto\mathcal{E}'[v]v$ and $v\mapsto \mathcal{E}[v]$ define equivalent norms in the homogeneous space, we also have $\mathcal{E}[u_n] \to \mathcal{E}[u]$. Combining this with weak-to-weak continuity of terms other than the Kirchhoff term of $I'$ (Lemma \ref{weaktoweakcontinuity}), we get $I'[u]=0$. Comparing $I'[u_n]u_n\to 0$ and $I'[u]u=0$, we get $\|V|u_n|^{p(x)}\|_1\to \|V|u|^{p(x)}\|_1$. By the equivalence of norm convergence and modular convergence and by the uniform convexity of $E$, $u_n\to u$ in $E$. \par
Using this method of proof, it is not necessary to show that $V_\infty=0$. Although this alternative method via $I'[u]u=0$ seems to be based on more common ideas, it is more roundabout than the original proof.
\end{remark}

\section{Proof of the main theorem}
In order to obtain a sequence of weak solutions of \eqref{cho}, we utilize a critical point theorem based on the concept of Krasnoselskii genus. 
\begin{definition}
Let $X$ be a Banach space and $\Sigma$ be the set of all the closed subsets of $X\setminus\{0\}$ which are symmetric with respect to the origin of $X$. For $A\in\Sigma$, we define
\[
\gamma(A)\coloneqq \inf\{n\in\mathbb{N}\mid \exists \varphi\in C(A,\mathbb{R}^n)\text{ s.t. }\forall u\in A;\varphi(-u)=-\varphi(u)\}.
\]
$\gamma(A)$ is called the Krasnoselskii genus of $A$. 
Furthermore, we set
\[
\Sigma_n\coloneqq \{A\in\Sigma\mid \gamma(A)\geq n\}.
\]
\end{definition}
As for the properties of genus, see \cite{CriticalValue}. We use the following version of the symmetric mountain pass lemma due to Ambrosetti-Rabinowitz \cite{SymmetricMountainPass}.
\begin{proposition}\label{symmetric mountain pass}
Let $X$ be an infinite-dimensional Banach space. Suppose that even functional $I\in C^1(X;\mathbb{R})$ bounded from below with $I(0)=0$ satisfies the $(PS)_c$ condition for $c<0$. Assume that for each $n\in\mathbb{N}$, there exists $A_n\in \Sigma_n$ such that $\displaystyle \sup_{A_n}I<0$. Then, each $\displaystyle c_n\coloneqq \inf_{A\in\Sigma_n}\sup_A I$ is a critical value of $I$ and $c_n\to 0$ ($n\to\infty$). \par
In particular, there exists a sequence of critical points $u_n\neq 0$ such that $I[u_n]\leq 0$, $u_n\to 0$ in $X$.
\end{proposition}
However, $I$ is not bounded from below, so we consider a truncated functional. We are now in a position to prove our main theorem.
\begin{proof}
We observe that
\begin{align*}
I[u]&\geq C_1\varrho_E(u)-C_2\max\{\|u\|^{2r^{+}}, \|u\|^{2r^{-}}\}\\
&\phantom{=}-C_3\max\{\|u\|^{(q_c)^{+}}, \|u\|^{(q_c)^{-}}\}-\varepsilon_W\cdot C_4\max\{\|u\|^{(q_W)^{+}}, \|u\|^{(q_W)^{-}}\} \\
&\geq \ell_\varepsilon (\varrho_E(u)),
\end{align*}
where
\begin{align*}
\ell_\varepsilon (t)&\coloneqq \min \{
C_1 t-C_2 t^{2r^{+}/p^{-}} -C_3 t^{(q_c)^{+}/p^{-}}-\varepsilon\cdot C_4 t^{(q_W)^{+}/p^{-}}, \\
&\phantom{\coloneqq \min\{\}} C_1 t-C_2 t^{2r^{-}/p^{+}} -C_3 t^{(q_c)^{-}/p^{+}}-\varepsilon\cdot C_4 t^{(q_W)^{-}/p^{+}}
\}.
\end{align*}
Since $(q_W)^{+}/p^{-}<1<\min\{2r^{-}/p^{+},(q_c)^{-}/p^{+}\}$, there exists $\varepsilon_0'>0$ so small that for any $\varepsilon\in (0,\varepsilon_0')$, there exist $t_{\varepsilon,0}$ and $t_{\varepsilon,1}$ with $0<t_{\varepsilon,0}<t_{\varepsilon,1}$ such that $\ell_\varepsilon (t)<0$ for $t\in (0, t_{\varepsilon,0})$, $\ell_\varepsilon (t)>0$ for $t\in (t_{\varepsilon,0}, t_{\varepsilon,1})$, $\ell_\varepsilon (t) <0$ for $t> t_{\varepsilon,1}$. To see this, it suffices to consider a small perturbation of the graph of $\ell_0$. \par
Let $\psi\in C^\infty(\mathbb{R})$ be such that $\psi(s)=1$ for $s\in [0, t_{\varepsilon_W,0})$ and $\psi(s)=0$ for $s> t_{\varepsilon_W,1}$. Define the truncated functional
\begin{align*}
\tilde{I}[u]&=\mathcal{M}(\mathcal{E}[u])+ \int_{\mathbb{R}^N} \frac{1}{p(x)}V(x)|u|^{p(x)}dx \\
&\quad\quad- \psi(\varrho_E(u))\int_{\mathbb{R}^N} \int_{\mathbb{R}^N} \frac{|u(x)|^{r(x)} |u(y)|^{r(y)}}{2r(x)|x-y|^{\alpha(x,y)}}dxdy-\psi(\varrho_E(u))\int_{\mathbb{R}^N} \frac{1}{q_c(x)}|u|^{q_c(x)}dx \\
&\quad\quad -\varepsilon_W \psi(\varrho_E(u))\int_{\mathbb{R}^N} \frac{1}{q_W(x)} W(x)|u|^{q_W(x)}dx.
\end{align*}
Then, $\tilde{I}$ is even, bounded from below and satisfies $(PS)_c$ condition for any $c<0$ and for any $\varepsilon_W\in (0,\min\{\varepsilon_0,\varepsilon_0'\})$ thanks to Lemma \ref{PS}. In fact, if $\tilde{I}[u]<0$, then since it necessarily holds that $\psi(\varrho_E(u))>0$, we get $\varrho_E(u) < t_{\varepsilon_W,1}$. Moreover, $\ell_{\varepsilon_W}(\varrho_E(u))<0$ implies $\varrho_E(u) < t_{\varepsilon_W,0}$ and thus $\tilde{I}[u]=I[u]$. \par
Next, let us show that for any $n\in\mathbb{N}$, there exists $\delta_n<0$ such that $\gamma(\tilde{I}^{-1}((-\infty,\delta_n]))\geq n$. 
Take a $n$-dimensional subspace $X_n$ of $E$. For each $u\in X_n\setminus\{0\}$, write $u=\rho v$ with $\rho\coloneqq \|u\|$. Since $X_n\cap \{\|u\|=1\}$ is compact, there exist $d_n,e_n>0$ such that 
\[
\sup_{v\in X_n\cap \{\|u\|=1\}}\left(\mathcal{M}(\mathcal{E}[v])+ \int_{\mathbb{R}^N} \frac{1}{p(x)}V(x)|v|^{p(x)}dx\right)\leq d_n,
\]
\[
\inf_{v\in X_n\cap \{\|u\|=1\}}\int_{\mathbb{R}^N}\frac{1}{q_W(x)}W(x)|v|^{q_W(x)}dx\geq e_n.
\]
For $0<\rho<\min\{t_{\varepsilon_W,0}^{1/p^{+}}, t_{\varepsilon_W,0}^{1/p^{-}}, 1\}$ and $\|v\|=1$, we can estimate $\tilde{I}$ as follows. 
\begin{align*}
\tilde{I}[\rho v]&\leq \rho^{p^{-}}\mathcal{M}(\mathcal{E}[v])+ \rho^{p^{-}}\int_{\mathbb{R}^N} \frac{1}{p(x)}V(x)|v|^{p(x)}dx \\
&\quad - \rho^{2r^{+}}\int_{\mathbb{R}^N} \int_{\mathbb{R}^N} \frac{|v(x)|^{r(x)} |v(y)|^{r(y)}}{2r(x)|x-y|^{\alpha(x,y)}}dxdy-\rho^{(q_c)^{+}}\int_{\mathbb{R}^N} \frac{1}{q_c(x)}|v|^{q_c(x)}dx \\
&\quad -\rho^{(q_W)^{+}}\varepsilon_W\int_{\mathbb{R}^N} \frac{1}{q_W(x)} W(x)|v|^{q_W(x)}dx \\
&\leq C_1 d_n \rho^{p^{-}}-C_2 e_n \rho^{(q_W)^{+}} \eqqcolon \delta_n.
\end{align*}
Here, although $C_2$ depends on $\varepsilon_W$, since we now consider fixed $\varepsilon_W$, there is no problem. 
Since $q_W^{+}<p^{-}$, there exists $\rho>0$ so small that the right hand side is strictly negative, that is, $\delta_n<0$. Now we have $X_n\cap\{\|u\|=\rho\}\subset \tilde{I}^{-1}((-\infty,\delta_n])$. Consequently, by the monotonicity of Krasnoselskii genus
\[
\gamma(\tilde{I}^{-1}((-\infty,\delta_n]))\geq \gamma(X_n\cap\{\|u\|=\rho\})=\gamma (S^{n-1})=n.
\]
Namely, $\tilde{I}^{-1}((-\infty,\delta_n]) \in\Sigma_n$. Hence, we obtain
\[
c_n\coloneqq\inf_{A\in\Sigma_n}\sup_{A} \tilde{I}\leq\delta_n<0.
\]
By Proposition \ref{symmetric mountain pass}, $\{c_n\}$ is a sequence of critical values of $\tilde{I}$. 
\end{proof}


\end{document}